\documentclass[12pt]{amsart}
\usepackage{amscd,amssymb,graphics}
\usepackage[hidelinks]{hyperref}
\usepackage{amsfonts}
\usepackage{amsthm}
\usepackage{amsmath}
\usepackage{enumerate}
\usepackage{hyperref}
\usepackage{gensymb}
\usepackage{yfonts}
\usepackage{amscd,amssymb}
\input xy
\xyoption{all}

\newtheorem{thm}{Theorem}[section]

\newtheorem{corol}[thm]{Corollary}
\newtheorem{lemma}[thm]{Lemma}

\newtheorem{defi}[thm]{Definition}
\newtheorem{remark}[thm]{Remark}

\theoremstyle{remark}
\newtheorem{example}[thm]{Example}
\newcommand{\quot}[2]{{\raisebox{.2em}{$#1$}\left/\raisebox{-.2em}{$#2$}\right.}}

\newcommand{\ben}{\begin{enumerate}}
	\newcommand{\een}{\end{enumerate}}
\newcommand{\bit}{\begin{itemize}}
	\newcommand{\eit}{\end{itemize}}

\newcommand{\neighbor}[2][ ]{\mathcal{N}_{#1}(#2)}
\DeclareMathOperator{\comp}{Comp}
\DeclareMathOperator{\conn}{Conn}

\ifdefined\FULLPROOFS
	\usepackage{color}
	\newcommand{\fullproof}{\begin{tiny}\color{blue}}
	\newcommand{\finishfullproof}{\end{tiny}}
\else
\let\fullproof=\iffalse
\let\finishfullproof=\fi
\fi

\def\R {{\Bbb R}}

\def\C {{\Bbb C}}
\def\N{{\Bbb N}}

\def\Z {{\Bbb Z}}

\def\F{{\mathbb{F}}}

\def\eps{{\varepsilon}}

\def\QED{\nobreak\quad\ifmmode\roman{Q.E.D.}\else{\rm Q.E.D.}\fi}

\begin{document}

	\title[]{Chinese Remainder Approximation Theorem}
	
	\author[]{Matan Komisarchik}
	\address{Department of Mathematics,
		Bar--Ilan University, 52900 Ramat--Gan, Israel}
	\email{komisan@macs.biu.ac.il}
	
	\date{June 26, 2016}
	
	\keywords{	Chinese Remainder Theorem, 
				Hyperspace uniformity,
				Interpolation,
				Supercomplete spaces,
				Topological ring,
				}

	\subjclass[2010]{	54H13, 
						54E15, 
						13A18, 
						13J10, 
						30E05  
					}
	
	\maketitle	
	\begin{abstract}
		We study a topological generalization of ideal co--maximality in topological rings and present some of its properties, including a generalization of the Chinese remainder theorem. Using the hyperspace uniformity, we prove a stronger version of this theorem concerning infinitely many ideals in supercomplete, pseudo--valuated rings. Finally we prove two interpolation theorems.
	\end{abstract}
	
	\setcounter{tocdepth}{1}
	
	\section{Introduction}	
	This paper studies topological versions of the \emph{Chinese remainder theorem -- CRT} using two main concepts: topological co--maximality and the hyperspace uniformity. After establishing both of these notions, we proceed by proving the \emph{Chinese remainder approximation theorem -- CRAT} (\ref{thm:CRAT}). Our final results will be derived from it.
	
	We begin by introducing the notion of \emph{topological ideal co--maximality -- TCM}, explaining what motivates our definition and presenting several examples. Then we show how some properties of co--maximality remain valid in the topological case. We also obtain a result resembling the second isomorphism theorem for co--maximal ideals in topological rings (\ref{thm:topological_second_iso_theorem}). Finally, we prove a direct extension of the CRT for finite families of ideals (\ref{thm:Finite_CRAT}).
	
	The hyperspace uniformity~\cite[p.~28]{Isbell} is used to study the case of infinite families of ideals. After a brief reminder of the basic definitions, we continue by studying topological co--maximality from the perspective of the hyperspace. From here on, our discussion is restricted to the class of pseudo--valuated rings. We show some approximation properties of those rings which will be used later to prove a strengthened version of the CRAT. Our main example is the ring of analytic functions over a domain in $\C$. This example will be used later to prove a statement about interpolation in infinite amount of points.
	
	We then prove the CRAT for compact families of pairwise TCM ideals in general topological rings. Shortly afterwards, we present a stronger version for supercomplete, pseudo--valuated rings and provide some applications. In particular, we will prove two known interpolation theorems:~\cite[Corollary~9 on p.~366]{Childs} and~\cite[Theorem~15.13\fullproof on pp.~304--305\finishfullproof]{Rudin}.
	
	\textbf{Acknowledgments:}
	I would like to thank Michael Megrelishvili, Menachem Shlossberg, Luie Polev, Tahl Nowik and Shahar Nevo for their valuable suggestions.
	\section{Preliminaries}	
	All topological spaces mentioned below are Hausdorff. 
	For every topological space $X$, we define $\conn(X)$ to be the set of all connected components of $X$. Also, if $A$ is a subset of $X$, then its closure is denoted by $\overline{A}$. The filter of neighborhoods at a given point $x \in X$ will be denoted by $\neighbor[X]
	{x}$ or simply $\neighbor{x}$ when no confusion can arise. Any uniform space $(X, \mu)$ will be denoted by $\mu X$.	If $Y$ is a uniform space, then $C(X,Y)$ is the space of all continuous functions from $X$ to $Y$ with the uniformity of compact convergence.
	\fullproof
~\cite[p. 229]{Kelley}
	\finishfullproof
	
	If $(G,+)$ is an (abelian) topological group 
	and $\eps$ is a neighborhood of the zero element, then $\frac{1}{n}\eps$ is any neighborhood of zero such that 
	\[\underbrace{\frac{1}{n}\eps + \dotsb +\frac{1}{n}\eps}_{\text{n times}} \subseteq \eps.\]
	All the rings contain an identity.
	
	Let $\hat{\C}$ be the Riemann sphere $\C \cup \{\, \infty \,\}$. Given a nonempty open set $\Omega \subset \hat{\C}$, we denote the topological ring of all analytic functions on $\Omega$ with the compact--open topology by $A(\Omega)$. By $\N$ we mean the set of all natural numbers including zero.
	Also, for any integer $n \in \N$, $f^{(n)}$ is the n'th derivative of $f$. More information on analytic functions can be found in~\cite{Rudin}.
	
	\section{Topological Co--Maximality}	
	Two integers are said to be \emph{co--prime} if their only common natural divisor is $1$, or equivalently, if every integer can be written as a sum of their products. Similarly, two ideals $I$ and $J$ of a ring $R$ are said to be \emph{co--maximal} if there is no proper ideal containing them both, or equivalently, if $I + J = R$. We make a 
	natural step generalizing this definition for topological rings. Instead of requiring $I + J$ to contain every element, we just want it to be dense in $R$. 
	\begin{defi}
		Let $R$ be a topological ring. We say that two ideals $I, J \unlhd R$ are \emph{topologically co--maximal (TCM)} and write $I \bot J$ if there is no \emph{closed} proper ideal containing them both. Equivalently, this can be formulated as: $\overline{I + J} = R$. \\
		We also say that a family of ideals $\mathcal{I}$ is \emph{pairwise TCM} and write $\bot \mathcal{I}$  if any two distinct members are TCM. 
	\end{defi}
	If two ideals are co--maximal then they are topologically co--maximal. The converse is also true when $R$ is discrete, but not in general. 	
	
	We present some examples where topological co--maximality can be easily characterized.
	\begin{example} \label{example:P_addic} 
		Consider $\Z$ with the $p$--adic topology. Two ideals $I,J$ in $\Z$ are TCM if and only if one of them is dense.
	\end{example}
	\begin{proof}
		 ($\Leftarrow$) Suppose that either $I$ or $J$ is dense in $\Z$. Without loss of generality, we can assume that it is $I$. Thus: $\overline{I + J} \supseteq \overline{I} = \Z$.
		 
		 ($\Rightarrow$) We will prove that if neither $I$ nor $J$ is dense then they are not TCM. Write $I = a \Z,\ J = b \Z$. Without loss of generality, we assume that $I$ and $J$ are closed. Note that all the ideals of $\Z$ are co--finite, hence $I$ and $J$ are open. By the definition of the $p$--adic topology, there exist $n, m \in \N$ such that $p^{n} \Z \subseteq I,\ p^{m} \Z \subseteq J$ meaning that $a \mid p ^{n}$ and $b \mid p^{m}$. Since $p$ is prime, there exist $n' \leq n, m' \leq m$ such that $a = p^{n'}, b = p^{m'}$. Because neither $I$ nor $J$ is dense, $n', m' \geq 1$. Thus,
		 \[
		   \overline{I + J} = 
		   \overline{ p^{n'} \Z + p^{m'}\Z} =
		   \overline{p^{\min(n', m')}\Z} = 
		   p^{\min(n', m')}\Z \subseteq
		   p \Z \subset
		   \Z. \qedhere
		 \]
	\end{proof}
	\begin{example}
		Let $R$ be a topological ring and suppose that $I$ and $J$ are ideals of $R$ such that $I$ is compact and $J$ is closed. In that case, $I$ and $J$ are topologically co--maximal if and only if they are co--maximal.
	\end{example}
	\begin{proof}
		By~\cite[Theorem~1.13\fullproof on p.~4\finishfullproof]{Ursul}, $I + J$ is closed and therefore it is dense if and only if it is equal to $R$ itself.
	\end{proof}
	\begin{example}
		In Banach algebras with identity, co--maximality of ideals is equivalent to topological co--maximality.
	\end{example}
	\begin{proof}
		By~\cite[Proposition~II.3 on p.~178]{Naimark}, the closure of any proper ideal in $R$ is also a proper ideal of $R$.
	\end{proof}
	
	Now we present some elementary properties of topological co--maximality. 
	\begin{lemma} \label{lemma:TCM_Properties}
		Let $R$ be a topological ring and let $I,J$ be two--sided ideals. 
		\ben
			\item $I$ and $J$ are TCM if and only if every neighborhood of $1$ contains an element of the form $i + j$ for $i \in I, j \in J$. \label{lemma:TCM_Properties:Identity}
			\item If $I$ and $J$ are TCM, then $IJ + JI$ is dense in $I \cap J$. Therefore, if $I$ and $J$ are also closed then $\overline{IJ + JI} = I \cap J$.\label{lemma:TCM_Properties:Denseness}
			\item If $I \subseteq J$ and $I \bot K$ then $J \bot K$. \label{lemma:TCM_Properties:Monotonicity}
			\item Let $J_{1} ,\dotsc,J_{n} \unlhd R$ be two--sided ideals.
			If $I \bot J_{i}$ for all $1 \leq i \leq n$, then $I \bot (J_{1} \dotsm J_{n})$, and therefore $I \bot (\bigcap_{i = 1}^{n} J_{i})$. \label{lemma:TCM_Properties:Intersection}			
		\een
	\end{lemma}
	\begin{proof}~\\
		\ben
		\item By definition, $I$ and $J$ are TCM if and only if $\overline{I + J} = R$. Because $\overline{I + J}$ is an ideal of $R$, this is equivalent to $1$ being an element of it. Finally, this happens if and only if every neighborhood of $1$ intersects $I + J$, meaning that it contains an element of the form $i + j$ for $i \in I, j \in J$. 
		\item Let $x$ be an element of $I \cap J$ and $\eps$ a neighborhood of zero.
		We need to show that $(x + \eps) \cap (IJ + JI) \neq \phi$. 
		Multiplication is continuous so there exists $\delta \in \neighbor{1}$ such that $\delta x \subseteq x + \eps$.
		By (\ref{lemma:TCM_Properties:Identity}), there exist $i \in I, j \in J$ such that $i + j \in \delta$. Thus, $(i + j) x \in \delta x  \subseteq x + \eps$. Recall that $x \in I \cap J$, so we get:
		\[(i + j) x  = 
		  \underbrace{i x}_{\in IJ} + \underbrace{j x}_{\in JI} \in
		  IJ + JI,
		\] 
		proving that $(i + j) x \in (x + \eps) \cap (IJ + JI) \neq \phi$.
		
		If we also assume that $I$ and $J$ are closed, then so is their intersection. Because $IJ + JI$ is dense in the closed set $I \cap J$, we conclude that \linebreak $\overline{IJ + JI} = I \cap J$.
		\item This follows directly from $\overline{J + K} \supseteq \overline{I + K} = R$.
		\item Ideals are closed under multiplication so it is sufficient to prove for $n = 2$ and the general case will be achieved by induction. 
		Let $\eps$ be a neighborhood of $1$. By (\ref{lemma:TCM_Properties:Identity}), it is enough to prove that $\eps$ contains an element of $I + J_{1} J_{2}$.
		
		There exists a neighborhood $\delta \in \neighbor{1}$ such that $\delta ^ {2} \subseteq \eps$ 
		Because $I \bot J_{1}$ and $I \bot J_{2}$, there exist $i_{1},i_{2} \in I, j_{1} \in J_{1}, j_{2} \in J_{2}$
		such that $i_{1} + j_{1}, i_{2} + j_{2} \in \delta$. Notice that
		\[(i_{1} + j_{1}) (i_{2} + j_{2}) 
		= \underbrace{i_{1} i_{2}}_{\in I} + \underbrace{i_{1} j_{2}}_{\in I} + \underbrace{j_{1} i_{2}}_{\in I} + \underbrace{j_{1} j_{2}}_{\in J_{1} J_{2}} \in I + J_{1} J_{2}.\]
		Moreover
		\[(i_{1} + j_{1}) (i_{2} + j_{2}) \in \delta^{2} \subseteq \eps.\]
		Hence
		\[(i_{1} + j_{1}) (i_{2} + j_{2}) \in (I + J_{1} J_{2}) \cap \eps \neq \phi.\]
		\een
	\end{proof}
	
	\begin{lemma} \label{lemma:Finite_Intersection_Continuity}
		Let $R$ be a topological ring and let $I_{1}, \dotsc, I_{n} \unlhd R$ be two--sided ideals of $R$.
		If $\{\,I_{k}\,\}_{k = 1}^{n}$ are pairwise co--maximal, then for any $\eps \in \neighbor{0}$ there exists $\delta \in \neighbor{0}$ such that $\bigcap\limits_{k = 1}^{n} (I_{k} + \delta) \subseteq \big( \bigcap\limits_{k = 1}^{n} I_{k} \big) + \eps$.
	\end{lemma}
	\begin{proof}
		We prove our claim via induction on $n$. The case of $n = 1$ is trivial. We prove the case of $n = 2$ separately. The ideals $I_{1}$ and $I_{2}$ are co--maximal so there exist $a \in I_{1}, b \in I_{2}$ such that $a + b = 1$. We pick a symmetric $\delta \in \neighbor{0}$ such that both $\delta$ and $a\delta$ are contained in  $\frac{1}{3} \eps$. Now suppose that $r \in \delta + I_{1}$ and $r \in \delta + I_{2}$. There exist $r_{1} \in I_{1}, r_{2} \in I_{2}$ such that $r - r_{1} \in \delta, r - r_{2} \in \delta$.
		Define 
		\[r' = b r_{1} + a r_{2} \in I_{2} I_{1} + I_{1} I_{2} \subseteq I_{1} \cap I_{2}.\]
		Notice that
		\[r - r' = r - b r_{1} - a r_{2} = r - r_{1} + r_{1} - b r_{1} - a r_{2} \]
		\[= (r - r_{1}) + (1 - b)r_{1} - a r_{2}
		 = (r - r_{1}) + a(r_{1} - r_{2}) \]
		\[= (r - r_{1}) + a(r_{1} - r) + a (r - r_{2}) 
		 \in \delta + a \delta + a \delta \]
		\[ \subseteq \frac{1}{3} \eps + \frac{1}{3} \eps + \frac{1}{3} \eps \subseteq \eps.\]
	    Thus, $r \in I_{1} \cap I_{2} + \eps$ and therefore $(I_{1} + \delta) \cap (I_{2} + \delta) \subseteq I_{1} \cap I_{2} + \eps$.
	    
		Now suppose that our claim is true for some $n \in \N$; we will prove it for $n + 1$.
		Let $I_{1}, \dotsc, I_{n}, I_{n + 1}$ be a family of pairwise co--maximal two--sided ideals and let $\eps$ be a symmetric neighborhood of zero. By the discrete case of Lemma~\ref{lemma:TCM_Properties}.\ref{lemma:TCM_Properties:Intersection}, $I := \bigcap\limits_{k = 1}^{n} I_{k}$ and $I_{n + 1}$ are co--maximal. By the case of $n = 2$, there exists a neighborhood $\delta'$ such that
		\[(I + \delta') \cap (I_{n + 1} + \delta') \subseteq
		 \big(\bigcap\limits_{k = 1}^{n + 1} I_{k}\big) + \eps.\]
		By the induction hypothesis, there exists a neighborhood $\delta''$ such that 
		\[\bigcap\limits_{k = 1}^{n} (I_{k} + \delta'') \subseteq
		 \big(\bigcap\limits_{k = 1}^{n} I_{k}\big) + \delta' = I + \delta'.\]
		By choosing $\delta := \delta' \cap \delta''$, we get the desired result. 
	\end{proof}
	
	Suppose that $R$ is a topological ring and $I$ and $J$ are two--sided ideals of $R$. Consider the homomorphism $\varphi\colon I \to \quot{(I + J)}{J}$ defined by $\varphi(a) := a + J$. By the second isomorphism theorem, the induced homomorphism $\widehat{\varphi}\colon \quot{I}{I \cap J} \to \quot{(I + J)}{J}$ is an isomorphism but generally not a topological isomorphism.
	\begin{thm} \label{thm:topological_second_iso_theorem}
		If $I$ and $J$ are co--maximal, then $\widehat{\varphi}$ is a topological isomorphism.
	\end{thm}
	\begin{proof}		
		By~\cite[Theorem 5.11\fullproof on p.35\finishfullproof]{Warner}, $\widehat{\varphi}$ is a topological isomorphism if and only if $\varphi$ is continuous and open.
		Clearly, $\varphi$ is continuous being the restriction to $I$ of the quotient map from $R$ to $\quot{R}{J}$. To prove that it is open, let $\eps$ be a symmetric neighborhood of zero. We need to prove that $\varphi(\eps \cap I)$ is a neighborhood of zero in $\quot{R}{J}$. Because $I$ and $J$ are co--maximal, there exist $a \in I, b \in J$ such that $a + b = 1$. There also exists a symmetric neighborhood $\delta \in \neighbor{0}$ such that $a \delta \subseteq \eps$. We will prove that $\delta + J \subseteq \varphi(\eps \cap I)$ thus proving that $\varphi(\eps)$ is a neighborhood of zero.
		Let $x$ be an element of $\delta$. We define $y := a x \in I$. Notice that $y = a x \in a \delta \subseteq \eps$ and therefore $y \in \eps \cap I$. Finally, 
		\[\varphi(y) = \varphi(a x) = a x + J = a x + bx + J = (a + b)x + J = x + J.\]
		This means that $x + J$ is a member of $\varphi(\eps \cap I)$, as required.
	\end{proof}
	
	We finish this section with an extension of the Chinese Remainder Theorem. 
	\begin{thm} \label{thm:Finite_CRAT}
		Let $R$ be a topological ring and let $I_{1}, \dotsc, I_{n} \unlhd R$ be two--sided ideals of $R$. Consider the map $\varphi\colon R \to \prod\limits_{k = 1}^{n} \quot{R}{I_{k}}$ defined by 
		$\varphi(r) := (r + I_{k})_{k = 1}^{n}$. 
		\ben
			\item If $\{\, I_{k} \,\}_{k = 1}^{n}$ are pairwise TCM, then the image of $\varphi$ is dense in $\prod\limits_{k = 1}^{n} \quot{R}{I_{k}}$. \label{thm:Finite_CRAT:denseness}
			\item If $\{\, I_{k} \,\}_{k = 1}^{n}$ are also algebraically pairwise co--maximal, then $\varphi$ is open onto its image.\label{thm:Finite_CRAT:openess}
		\een
	\end{thm}
	\begin{proof}~\\
		\ben
		
			\item Let $\eps \in \neighbor[R]{0}$ be a symmetric neighborhood and $r_{1}, \dotsc, r_{n}$ be elements of $R$.
			We need to find $r \in R$ such that for all $1 \leq k \leq n:$ \[r + I_{k} \subseteq r_{k} + I_{k} + \eps.\]
			Notice that $\{\, I_{k} \,\}_{k = 1}^{n}$ are pairwise TCM so for any $1 \leq k \neq l \leq n$ we have $I_{k} \bot I_{l}$. For all $1 \leq k \leq n$, define 
			\[E_{k} := \bigcap \limits_{l \neq k} I_{l}.\]
			By Lemma~\ref{lemma:TCM_Properties}.\ref{lemma:TCM_Properties:Intersection}, we conclude that:
			\[\forall 1 \leq k \leq n: I_{k} \bot E_{k}.\]
			Thus, for any $1 \leq k \leq n$, there exist $a_{k} \in I_{k}, b_{k} \in E_{k}$ such that 
			\[a_{k} + b_{k} \in r_{k} + \eps. \]
			Notice that for any $1 \leq k \neq l \leq n$:
			\[b_{k} + I_{k} = a_{k} + b_{k} + I_{k} \subseteq r_{k} + I_{k} + \eps\]
			\[b_{k} + I_{l} = 0 + I_{l}. \]
			We define $r := \sum\limits_{k = 1}^{n} b_{k}$. As a consequence of the last two equations we have that for any $1 \leq k \leq n$:
			\[r + I_{k} = \sum\limits_{l = 1}^{n} b_{l} + I_{k} 
			= b_{k} + \sum\limits_{l \neq k} b_{l} + I_{k} \in r_{k} + I_{k} + \eps,\]
			as required.
			
			\item The map $\varphi$ is a homomorphism so it is enough to show that the image of any neighborhood of zero is also a neighborhood of zero. Let $\eps \in \neighbor{0}$. By Lemma~\ref{lemma:Finite_Intersection_Continuity}, there exists a neighborhood $\delta$ satisfying 
			\[\bigcap\limits_{k = 1}^{n} (I_{k} + \delta) \subseteq \big( \bigcap\limits_{k = 1}^{n} I_{k} \big) + \eps = \ker \varphi + \eps.\]
			We claim that  $\varphi(R) \cap \prod\limits_{k = 1}^{n} (I_{k} + \delta) \subseteq \varphi(\eps)$ making $\varphi(\eps)$ a neighborhood of zero in $\varphi(R)$. This is easily seen since if $(r + I_{k})_{k = 1}^{n}$ is an element of $\prod\limits_{k = 1}^{n} (I_{k} + \delta)$, then $r$ belongs to $\bigcap\limits_{k = 1}^{n} (I_{k} + \delta)$. By the choice of $\delta$, $r$ is also contained in $\ker \varphi + \eps$ and therefore $\varphi(r) = (r + I_{k})_{k = 1}^{n} \in \varphi(\eps)$, completing this proof.
		\een
	\end{proof}
	
	\section{The hyperspace uniformity} \label{sec:Hyperspace}
	
	So far, we have only discussed the case of finitely many ideals. In this section, we will use the hyperspace uniformity (as described in~\cite[p.~28]{Isbell}) to extend our results to some infinite families of ideals. 
	
	We think of uniformity in terms of entourages like in~\cite[Chapter~6]{Kelley}. If $\mu X$ is a uniform space, $U \in \mu$ is an entourage and $A$ is a subset of $X$, then:
	\[U[A] := \{\, x \in X | \ \exists \ a \in A: (a,x) \in U \,\}. \]
	The \emph{hyperspace} of $X$ (denoted $H(X)$) 
	 is the set of all nonempty, closed subsets in $X$ endowed with the \emph{hyperspace uniformity}: Given an entourage $U \in \mu$, a corresponding entourage $H(U)$ is defined on $H(X)$ as
	\[H(U) := \{\, (A,B) \in H(X) \times H(X) \mid A \subseteq U[B] \text{ and } B \subseteq U[A] \,\}. \]
	The hyperspace uniformity is the one with $\{\, H(U) \,\}_{U \in \mu}$ as a basis.
	
	Let $R$ be a topological ring. In this case, it will be easier for us to consider neighborhoods of zero instead of the entourages they induce. Thus, for any \linebreak $\eps \in \neighbor{0}$ we define
	\[H(\eps) := \{\, (A,B) \in H(X) \times H(X) \mid A \subseteq B + \eps \text{ and } B \subseteq A + \eps \,\}. \]
	
	The hyperspace uniformity is the one with $\{\, H(\eps) \,\}_{\eps \in N(0)}$ as a basis. 
	We will be interested in a particular subset of the hyperspace, namely the set $\mathcal{L}(R)$ of all closed two--sided ideals. Note that $\mathcal{L}(R)$ is a closed subspace of the hyperspace (Lemma~\ref{lemma:Modules_Are_Closed}).	
	\begin{lemma} \label{lemma:Hyperspace_Similarities}
		Suppose that $R$ is a topological ring.
		\ben
			\item If $R$ is metrizable, so is $\mathcal{L}(R)$. If $R$ is also complete, then so is $\mathcal{L}(R)$.
			\item If $R$ is compact, so is $\mathcal{L}(R)$.
		\een
	\end{lemma}
	\begin{proof}
		By Theorems 48 and 49 of~\cite[pp.~30--31]{Isbell}, if $R$ is metrizable [and complete] or compact, then so is $H(R)$. Therefore, both claims are a consequence of $\mathcal{L}(R)$ being a closed subspace of $H(R)$ (Lemma~\ref{lemma:Modules_Are_Closed}).
	\end{proof}

	Recall that $\mathcal{L}(R)$ is also a lattice, with the meet and join operations:
	\[\wedge: \mathcal{L}(R) \times \mathcal{L}(R) \rightarrow \mathcal{L}(R); \: \wedge(A,B) = A \cap B\]
	\[\vee: \mathcal{L}(R) \times \mathcal{L}(R) \rightarrow \mathcal{L}(R); \: \vee(A,B) = \overline{A + B}.\]
	\fullproof
	If $A$ and $B$ are closed, then so are $A \cap B$ and $\overline{A + B}$ meaning that both operations are defined. Now, to prove that these operations are really the meet and join, suppose that $C, D \in \mathcal{L}(R)$ such that $C \subseteq A, B \subseteq D$.
	
	By definition, $C \subseteq A, B$ implies $C \subseteq A \cap B := A \wedge B$. Moreover, $A, B \subseteq D$ implies that $A + B \subseteq D$. Since $D$ is closed, we conclude that $A \vee B :=\overline{A + B} \subseteq \overline{D} = D$. This proves the claim.
	
	\finishfullproof	
	It is worth noting that these functions are defined on the entire hyperspace for any topological ring.
	\begin{lemma} \label{lemma:Join_Is_Continuous}
		Let $R$ be a topological ring. The join operation is uniformly continuous on the entire hyperspace $H(R)$.
	\end{lemma}
	\begin{proof}
		Let $\eps \in \neighbor{0}$. We need to find $\delta \in \neighbor{0}$ such that if $(A_{1},A_{2}) \in H(\delta)$ and $(B_{1},B_{2}) \in H(\delta)$, then $(A_{1} \vee B_{1},A_{2} \vee B_{2}) \in H(\eps)$.\\
		Choose $\delta = \frac{1}{3} \eps$. Suppose that
		\ben
			\item $A_{1} \subseteq A_{2} + \delta$ \label{lemma:Join_Is_Continuous:A1_near_A2}
			\item $A_{2} \subseteq A_{1} + \delta$ \label{lemma:Join_Is_Continuous:A2_near_A1}
			\item $B_{1} \subseteq B_{2} + \delta$ \label{lemma:Join_Is_Continuous:B1_near_B2}
			\item $B_{2} \subseteq B_{1} + \delta$. \label{lemma:Join_Is_Continuous:B2_near_B1} 
		\een
		By adding suitable parts we get:
		\[A_{1} + B_{1} \subseteq A_{2} + B_{2} + 2 \delta \]
		\[A_{2} + B_{2} \subseteq A_{1} + B_{1} + 2 \delta. \]
		By~\cite[Theorem~3.3\fullproof on p.~20\finishfullproof]{Warner}, $\overline{C} = \bigcap \{\, C + V \mid V \in N(0) \,\}$ for any subset $C$ of $R$. 
		In particular, $\overline{C} \subseteq C + V$ for any neighborhood $V$ of zero. Therefore,	
		\[\overline{A_{1} + B_{1}} \subseteq \overline{A_{2} + B_{2} + 2 \delta} \subseteq A_{2} + B_{2} + 3 \delta \subseteq A_{2} + B_{2} + \eps\]
		\[\overline{A_{2} + B_{2}} \subseteq \overline{A_{1} + B_{1} + 2 \delta} \subseteq A_{1} + B_{1} + 3 \delta \subseteq A_{1} + B_{1} + \eps.\]
		By definition, $(A_{1} \vee B_{1},A_{2} \vee B_{2}) \in H(\eps)$.
	\end{proof}
	
	\begin{corol} \label{corol:TCM_Is_Closed}
		Let $R$ be a topological ring. Topological co--maximality is a closed relation on $\mathcal{L}(R)$, meaning that $\{\, (A,B) \in \mathcal{L}(R) \times \mathcal{L}(R) \ | \ A \bot B \,\}$ is closed in $\mathcal{L}(R) \times \mathcal{L}(R)$. 
	\end{corol}
	\begin{proof}
		Let $\{\, X_{\alpha} \,\}_{\alpha \in A}$ and $\{\, Y_{\alpha} \,\}_{\alpha \in A}$ be two converging nets (with respect to the hyperspace uniformity). 
		We need to show that if $X_{\alpha} \bot Y_{\alpha}$ for any $\alpha \in A$, then $(\lim X_{\alpha}) \bot (\lim Y_{\alpha})$.
		Notice that $X_{\alpha} \bot Y_{\alpha}$ if and only if $X_{\alpha} \vee Y_{\alpha} = R$. 
		The join operation is continuous by Lemma~\ref{lemma:Join_Is_Continuous} and therefore
		\[(\lim X_{\alpha}) \vee (\lim Y_{\alpha}) = \lim (X_{\alpha} \vee Y_{\alpha}) = \lim R = R. \]
		Thus $(\lim X_{\alpha}) \bot (\lim Y_{\alpha})$.
	\end{proof}
	
	The following corollary gives a convenient characterization of compact, pairwise TCM families of ideals.
	\begin{corol} \label{corol:Pairwise_TCM_Families}
		Let $R$ be a topological ring. If $\mathcal{I}$ is a family of closed, pairwise TCM ideals of $R$, then it has at most one non isolated point ($R$ itself). 
		In particular, if $\mathcal{I}$ is compact, then for any neighborhood of zero $\eps$, there exists a finite set $\mathcal{F} \subseteq \mathcal{I}$ such that
		\[\forall I \in \mathcal{I} \setminus \mathcal{F}: R = I + \eps.\]
	\end{corol}
	\begin{proof}
		Suppose that $I \neq R$ is some ideal in $\mathcal{I}$. By Corollary~\ref{corol:TCM_Is_Closed}, the set \linebreak $\mathcal{C} := \{\, J \unlhd R \mid I \bot J  \,\}$ is closed in the hyperspace of $R$. Consequently, the set $\mathcal{C} \cap \mathcal{I}$ is a closed subspace of $\mathcal{I}$. Because $\mathcal{I}$ is pairwise TCM, $J$ belongs to $\mathcal{C}$ for any $I \neq J \in \mathcal{I}$. 
		Also, $I$ is not TCM with itself since $I \neq R$. Thus, $\mathcal{C} \cap \mathcal{I} = \mathcal{I} \setminus \{\, I \,\}$ is closed in $\mathcal{I}$ making $I$ an isolated point in $\mathcal{I}$.
		
		Now suppose that $\mathcal{I}$ is also compact. 
		For any $I \in \mathcal{I}$, we define $U_{I} := \{\, I \,\}$ if $I \neq R$ and $U_{I} := H(\eps)[R]$ otherwise. Because any $I \neq R$ is isolated, the set $\{\, U_{I}  \,\}_{I \in \mathcal{I}}$ is an open cover of $\mathcal{I}$. Then there exists a finite subcover, say $U_{I_{1}}, \dotsc, U_{I_{n}}$. If neither of $\{\, I_{k} \,\}_{k=1}^{n}$ is equal to $R$, then $\mathcal{I}$ is finite since
		\[|\mathcal{I}| = 
		  \big|\bigcup\limits_{k = 1}^{n} U_{I_{k}}\big| \leq 
		  \bigcup\limits_{k = 1}^{n} |U_{I_{k}}| =
		  \bigcup\limits_{k = 1}^{n} |\{\, I_{k} \,\}| =
		  n < \infty.\]
		In this case, our claim is trivial. Alternatively, $R = I_{k}$ for some $1 \leq k \leq n$, and therefore any $I \in \mathcal{I} \setminus \{\, I_{l} \,\}_{l \neq k}$ belongs to $H(\eps)[R]$, as required.
		\fullproof
		Note that:
		\[
		I \in H(\eps)[R] \Rightarrow
		R \subseteq I + \eps
		\]
		\finishfullproof
	\end{proof}
	
	\begin{lemma} \label{lemma:pairwise_TCM_closure}
		Let $R$ be a topological ring. If $\mathcal{I}$ is a family of pairwise TCM ideals of $R$, then so is its closure in $\mathcal{L}(R)$.
	\end{lemma}
	\begin{proof}
		Suppose that $I$ and $J$ are two distinct elements in the closure of $\mathcal{I}$. There exist converging nets $\{\, I_{\alpha} \,\}_{\alpha \in A}, \{\, J_{\alpha} \,\}_{\alpha \in A} \subseteq \mathcal{I}$ such that $I_{\alpha} \to I, J_{\alpha} \to J$. 
		There also exist disjoint neighborhoods $U \in \neighbor[\mathcal{I}]{I}, V \in \neighbor[\mathcal{I}]{J}$. By definition, $I_{\alpha}$ and $J_{\alpha}$ are eventually in $U$ and $V$ respectively. Since $U$ and $V$ are disjoint and $\mathcal{I}$ is pairwise TCM, we have that $I_{\alpha}$ and $J_{\alpha}$ are eventually TCM. By Corollary~\ref{corol:TCM_Is_Closed}, $I = \lim\limits_{\alpha \in A} I_{\alpha}$ and $J = \lim\limits_{\alpha_\in A} J_{\alpha}$ are TCM too.
	\end{proof}
	
	Topologizing certain families of ideals is not a new concept. 
	The structure theory of commutative Banach algebras with identities over $\C$ is probably the best known example. If $R$ is a commutative, semisimple Banach algebra with identity over $\C$, then there is a natural compact topology on the set of its maximal ideals $\textswab{m}$. Moreover, $R$ is topologically isomorphic to the ring of continuous functions over $\textswab{m}$~\cite[Theorem~4 on p.~197]{Naimark}.
	
	By~\cite[Proposition~II.4 on p.~178]{Naimark}, any maximal ideal $M$ in a Banach algebra is closed so we could also look at $\textswab{m}$ with the topology induced from the hyperspace. One might expect the two topologies to be similar but they differ significantly.
	\begin{example}	
		Let $K$ be a compact space. Consider the topological ring $R$ of continuous complex valued functions on $K$. If we furnish the set of all maximal ideals of $R$ with the topology induced from the hyperspace, then it is discrete.
	\end{example}
	\begin{proof}
		By~\cite[Example~2 on p.~194]{Naimark}, any maximal ideal of $R$ is of the form 
		\[M_{x} := \{\, f \in R \mid f(x) = 0 \,\}\]
		for some $x \in K$. Let $x_{0} \in K$ be any point of $K$. We will prove that $M_{x_{0}}$ is isolated in $\textswab{m}$.
		Consider the neighborhood $U := \{\, f \in R \mid \forall x \in K: |f(x)| < \frac{1}{2} \,\}$ of zero in $R$.
		Suppose that there exists $x_{0} \neq x_{1} \in K$ such that $(M_{x_{0}}, M_{x_{1}}) \in H(U)$.
		By the Tietze extension theorem~\cite[Theorem~O on p.~242]{Kelley}, 
		there exists a function $f \in R$ such that $f(x_{1}) = 0, f(x_{0}) = 1$. By definition, $f \in M_{x_{1}}$ and since $(M_{x_{0}}, M_{x_{1}}) \in H(U)$, there also exists $g \in M_{x_{0}}$ satisfying $|f(x) - g(x)| < \frac{1}{2}$ for all $x \in K$. In particular, $|1 - g(x_{0})| < \frac{1}{2}$ and therefore $|g(x_{0})| > \frac{1}{2}$. However, $g \in M_{x_{0}}$ hence $g(x_{0}) = 0$. This contradiction proves that $M_{x_{0}}$ is isolated in $\textswab{m}$.
	\end{proof}
	
	\begin{defi}\cite{Isbell}
		A uniform space is \emph{supercomplete} if its hyperspace is complete. 
	\end{defi}
	\begin{example}
		Any complete metric space is supercomplete~\cite[Theorem~48 on p.~30]{Isbell}. 
		Compact spaces are also supercomplete by Lemma~\ref{lemma:Hyperspace_Similarities}.
	\end{example}
	\begin{lemma} \label{lemma:Image_Of_Supercomplete_Ring}
		Let $R$ and $S$ be two topological rings and let $\varphi\colon R \to S$ be a surjective, continuous, open homomorphism. If $R$ is supercomplete then so is $S$.
	\end{lemma}
	\begin{proof}
		Let $\{\, X_{\alpha} \,\}_{\alpha \in A}$ be a Cauchy net of subsets of $S$. Consider $Y_{\alpha} := \varphi^{-1}(X_{\alpha})$. We claim that $\{\, Y_{\alpha} \,\}_{\alpha \in A}$ is also a Cauchy net. Let $\eps$ be a neighborhood of zero in $R$. We need to find $\alpha_{0} \in A$ such that for any $\alpha_{0} \leq \alpha, \beta \in A$: 
		\[ Y_{\alpha} \subseteq Y_{\beta} + \eps \text{ and } Y_{\beta} \subseteq Y_{\alpha} + \eps. \]
		Since $\varphi$ is open, $\varphi(\eps)$ is a neighborhood of zero in $S$. Because $\{\, X_{\alpha} \,\}_{\alpha \in A}$ is a Cauchy net, there exists $\alpha_{0} \in A$ such that for any $\alpha_{0} \leq \alpha, \beta \in A$:
		\[ X_{\alpha} \subseteq X_{\beta} + \varphi(\eps) \text{ and } X_{\beta} \subseteq X_{\alpha} + \varphi(\eps). \]
		Thus,
		\[ \varphi^{-1}(X_{\alpha}) \subseteq \varphi^{-1}(X_{\beta} + \varphi(\eps)) \text{ and } \varphi^{-1}(X_{\beta}) \subseteq \varphi^{-1}(X_{\alpha} + \varphi(\eps)). \]
		Note that for any set $Z \subseteq S$
		\[\varphi^{-1}(Z + \varphi(\eps)) = \varphi^{-1}(Z) + \eps.\]
		Therefore,
		\[ Y_{\alpha} \subseteq Y_{\beta} + \eps \text{ and } Y_{\beta} \subseteq Y_{\alpha} + \eps, \]
		proving our claim. Since $R$ is supercomplete, $\left\{\, Y_{\alpha} \,\right\}_{\alpha \in A}$ converges to some $Y \subseteq R$.
		
		We now claim that $\left\{\, X_{\alpha} \,\right\}_{\alpha \in A}$ converges to $\varphi(Y)$. Let $\theta$ be a neighborhood of zero in $S$. Since $\varphi$ is continuous, $\varphi^{-1}(\theta)$ is a neighborhood of zero in $R$. Because $\left\{\, Y_{\alpha} \,\right\}_{\alpha \in A}$ converges to $Y$, there exists $\alpha_{0}$ such that for any $\alpha_{0} \leq \alpha \in A$:
		\[Y \subseteq Y_{\alpha} + \varphi^{-1}(\theta) \text{ and } 
		  Y_{\alpha} \subseteq Y + \varphi^{-1}(\theta). \]
		Recall that $\varphi$ is surjective so by applying $\varphi$ we get:
		\[
		\varphi(Y) \subseteq X_{\alpha} + \theta \text{ and }
		X_{\alpha} \subseteq \varphi(Y) + \theta. 
		\]
		This proves that $\left\{\, X_{\alpha} \,\right\}_{\alpha \in A}$ converges and therefore $S$ is supercomplete.
	\end{proof}
	
	\subsection{Pseudo--Valuated Rings}
	To prove some of our stronger results, we will need to restrict our discussion to a specific class of topological rings, namely, pseudo-valuated rings as described in~\cite{Arens}. We will shortly see that it is a fairly large class containing most of the examples we will be interested in.
	\begin{defi} \label{def:pseudo_valuated_ring}~\cite{Arens}
		Let $R$ be a ring. A \emph{pseudo--valuation} in $R$ is a nonnegative real--valued function $V$ such that for all $x, y \in R$ we have:
		\ben
		\item $V(x + y) \leq V(x) + V(y)$. \label{def:pseudo_valuated_ring:sub_additive}
		\item $V(x y) \leq V(x) V(y)$. \label{def:pseudo_valuated_ring:sub_multiplicative}
		\item $V(-x) = V(x)$. \label{def:pseudo_valuated_ring:symmetry}
		\item $V(0) = 0$. \label{def:pseudo_valuated_ring:seperation}
		\een
		The ring $R$ together with the family $\mathcal{V}$ of pseudo--valuations is said to be \emph{pseudo--valuated} if $V(x) = 0$ for all $V \in \mathcal{V}$ implies $x = 0$.
		For any $V \in \mathcal{V}$ and $\eps > 0$, the \emph{open $V$--ball of radius $\eps$ around $r_{0}$} is defined as:
		\[B_{V}(r_{0}, \eps) := \left\{\, r \in R \mid V(r - r_{0}) < \eps  \,\right\}. \]
		When $r_{0} = 0$, we will simply write $B_{V}(\eps)$.
	\end{defi}	
	Note that if $(R, \mathcal{V})$ is a pseudo valuated ring, then $\left\{\, B_{V}(\eps) \,\right\}_{\eps > 0, V \in \mathcal{V}}$ is a fundamental system of neighborhoods of zero for a ring topology on $R$.
	
	\begin{defi}~\cite{Naimark}
		A \emph{normed $\F$--algebra} (for $\F = \R, \C$) is an $\F$--algebra $R$ which is also a normed vector space and its norm satisfies $\lvert 1 \rvert = 1$ and $\lvert xy \rvert \leq \lvert x \rvert \lvert y \rvert$ for any $x, y \in R$.
	\end{defi}
	\begin{remark}
		Any valuated ring is clearly pseudo--valuated. In particular, any normed algebra is also pseudo--valuated.
	\end{remark}
	\begin{example} \label{example:Function_are_pseudo_valuated}
		Suppose that $R$ is a pseudo--valuated ring and $T$ is any topological space. The ring $C(T, R)$ of continuous functions from $T$ to $R$ furnished with the compact--open topology is also pseudo--valuated.
	\end{example}
	\begin{proof}
		Denote by $\tau$ the topology of $T$.
		Suppose that $\mathcal{V}$ is the family of pseudo--valuations on $R$. For any pseudo--valuation $V \in \mathcal{V}$ and any compact subset $K \subseteq T$, we define:
		\[ \forall f \in C(T, R): V_{K} (f) := \sup \limits_{x \in K} V(f(x)).\]
		Note that $V \circ f$ is continuous with respect to $\tau$ and $K$ is compact in $T$ so its image is a compact subset of $R$, hence it admits a supremum.
		This is clearly a pseudo--valuation on $C(T, R)$.
		\fullproof
		Axioms (\ref{def:pseudo_valuated_ring:seperation}), (\ref{def:pseudo_valuated_ring:sub_additive}) and (\ref{def:pseudo_valuated_ring:symmetry}) are obvious. To prove (\ref{def:pseudo_valuated_ring:sub_multiplicative}) note that:
		\[
		V_{K}(f g) := 
		\sup\limits_{x \in K} V (f(x) g(x)) \leq 
		\sup\limits_{x \in K} V(f(x)) V(g(x)) \leq
		\left(\sup\limits_{x \in K} V(f(x))\right) \left(\sup\limits_{x \in K} V(g(x))\right) =
		V_{K}(f) V_{K}(g).
		\]
		\finishfullproof		
		Moreover, the family $\mathcal{V}_{T} :=\left\{\, V_{K} \,\right\}_{V \in \mathcal{V}}$, where $K$ runs over all compact subsets of $T$, generates the compact--open topology on $C(T, R)$.
		\fullproof
		First we will see that the compact--open topology is weaker than that generated by $\left\{\, V_{K} \,\right\}_{K \in \comp(T)}$. Suppose that $K$ is a compact subset of $T$ and $U$ is some open neighborhood of zero in $R$. There exists some $V \in \mathcal{V}$ and $\eps > 0$ such that $B_{V}(\eps) \subseteq U$. Thus:
		\[B_{V_{K}}(\eps) \subseteq \left\{\, f \in C(T, R) \mid \forall x \in K: f(x) \in U  \,\right\} \]
		proving that the compact--open topology is indeed weaker.
		Conversely, let $\eps > 0$ and $K \in \comp(T)$. We show that there exists some set $U \subseteq B_{V_{K}}(\eps)$ which is open in the compact--open topology. Define:
		\[
		U := 
		\left\{\,  f \in C(T, R) \mid \forall x \in K: f(x) \in B_{V}(\eps) \,\right\}.
		\]
		Note that for any $f \in U$, $V_{K}(f) := \sup \limits_{x \in K} V(f(x)) \leq \eps$ since $f(x) \in B_{V}(\eps)$. Thus $U \subseteq B_{V_{K}}$, as required.
		\finishfullproof
	\end{proof}
	\begin{lemma}
		Let $R$ be a ring and let $V$ be a pseudo--valuation on $R$. The following inclusions hold for every $\eps, \delta > 0$:
		\ben
			\item $B_{V}(\eps) + B_{V}(\delta) \subseteq B_{V}(\eps + \delta). $
			\item $B_{V}(\eps) B_{V}(\delta) \subseteq B_{V}(\eps \delta). $
		\een
	\end{lemma}
	\begin{proof} 
		Directly follows from axioms (\ref{def:pseudo_valuated_ring:sub_additive}) and (\ref{def:pseudo_valuated_ring:sub_multiplicative}) of Definition~\ref{def:pseudo_valuated_ring}.
	\end{proof}
	
	\begin{defi}
		Suppose that $V$ is a pseudo--valuation on the ring $R$ and $\eps > 0$. A subset $A \subseteq R$ is \emph{$(V, \eps)$--dense} in $B \subseteq R$, if $B \subseteq A + B_{V}(\eps)$.
		If $A$ is $(V, \eps)$--dense in $B$ for all $\eps > 0$, then it is said to be \emph{$V$--dense} in $B$.
		If $B = R$, then we will simply say $(V, \eps)$--dense or $V$--dense.
	\end{defi}
	Now we will discuss some approximation properties of pseudo--valuated rings which will be useful for us later.
	\begin{lemma} \label{lemma:pseudo_valuated_ceil_size}
		Let $R$ be a ring and let $V$ be a pseudo--valuation on $R$. If $I$ is a $(V, \delta)$--dense left ideal for some $0 < \delta < 1$, then $I$ is $V$--dense.
	\end{lemma}
	\begin{proof}
		Let $r$ be any element of $R$ and $\eps > 0$. We need to find $r' \in I$ such that $V(r - r') < \eps$.
		Since $I$ is $(V, \delta)$--dense, there exists $a \in I$ such that $V(1 - a) < \delta$.
		Let $r_{0} := 0$. Now, we recursively define
		\[r_{n + 1} = r_{n} + (r - r_{n}) a.\]
		By induction, we prove that $V(r - r_{n}) \leq \delta^{n} V(r)$ and that $r_{n} \in I$. 
		For $n = 0$ it is trivial. Now suppose that our claim is true for some $n \in \N$; we will prove it for $n + 1$.
		\[V(r - r_{n + 1}) = V(r - (r_{n} + (r - r_{n})a) )= V( r - r_{n} - (r - r_{n})a)) 
		   = V((r - r_{n})(1 - a)) \]
		\[\leq V(r - r_{n}) V(1 - a) \leq \delta \delta^{n} V(r) = \delta^{n + 1} V(r).\]
		Moreover, $r_{n + 1}$ clearly belongs to $I$ because $r_{n}$ and $a$ do.
		Recall that $\delta < 1$ and therefore, there exists $n_{0} \in \N$ such that $\delta^{n_{0}} < \frac{\eps}{V(r)}$. Now for $r' := r_{n_{0}}$ we get:
		\[ V(r - r') = V(r - r_{n_{0}}) \leq \delta^{n_{0}} V(r) < \frac{\eps}{V(r)} V(r) = \eps,\]
		proving that $I$ is $V$--dense in $R$.
	\end{proof}
	
	\begin{defi}
		If $V$ is a pseudo--valuation on a ring $R$, then two ideals are \emph{TCM with respect to $V$} if their sum is $V$--dense in $R$.
	\end{defi}
	Note that two ideals in a pseudo--valuated ring are TCM if and only if they are TCM with respect to any individual pseudo--valuation.
	\begin{lemma} \label{lemma:co_maximal_product}
		Let $R$ be a ring and $V$ be a pseudo--valuation on $R$. Suppose that $I$ and $J$ are two--sided ideals of $R$. If $I$ and $J$ are TCM with respect to $V$ and $J$ is $(V, \delta)$--dense for some $0 < \delta < 1$, then $I \cap J$ is $V$--dense in $I$.
	\end{lemma}
	\begin{proof}
		Let $x$ be any element of $I$ and $\eps > 0$. We will find $y \in I \cap J$ such that $V(x - y) < \eps$. Since $I$ and $J$ are TCM with respect to $V$,
		there exist $a \in I, b \in J$ such that $V(1 - (a + b)) < \frac{\eps}{2 V(x)}$. By Lemma~\ref{lemma:pseudo_valuated_ceil_size}, $J$ is $V$--dense in $R$ so there exists $x' \in J$ such that $V(x - x') < \frac{\eps}{2V(a)}$. We define $y = a x' + b x \in I \cap J$.
		Notice that:
		\[ V(x - y) =
		   V(x - (a x' + bx)) = 
		   V(x - b x - a x + a x - a x') = \]
		\[ V((1 - (b + a))x + a (x - x')) \leq
		   V(1 - (b + a)) V(x) + V(a) V(x - x') \leq\]
		\[ \frac{\eps}{2 V(x)} V(x) + V(a) \frac{\eps}{2V(a)} =
		   \eps.\]
		   
		This concludes the proof.
	\end{proof}
	\begin{corol} \label{corol:co_maximal_product}
		Let $R$ be a topological ring and let $V$ be a pseudo--valuation on $R$. Suppose that $I_{1}, \dotsc, I_{n}$ are two--sided ideals of $R$, pairwise TCM with respect to $V$. If they are also $(V, \delta)$--dense, for some $0 < \delta < 1$, then $\bigcap_{i = k}^{n} I_{k}$ is $V$--dense.
	\end{corol}
	\begin{proof}
		By induction on $n$. For $n = 1$, we simply apply Lemma~\ref{lemma:pseudo_valuated_ceil_size}. Now suppose that our claim is true for some natural number $n \in \N$; we will prove it for \linebreak $n + 1$. Suppose that $I_{1}, \dotsc, I_{n}, I_{n + 1}$ are $(V, \delta)$--dense for some $0 < \delta < 1$ and let $I = \bigcap_{k = 1}^{n} I_{k}, J = I_{n + 1}$. Note that $I$ and $J$ are TCM by Lemma~\ref{lemma:TCM_Properties}. Applying Lemma~\ref{lemma:co_maximal_product} we conclude that $I \cap J = \bigcap_{k = 1}^{n + 1}I_{k}$ is $V$--dense in $I$. The latter is $V$--dense by the induction hypothesis. It is not hard to see that as a consequence $\bigcap_{k = 1}^{n + 1}I_{k}$ is $V$--dense.
		\fullproof
		We have shown that $\bigcap_{k = 1}^{n + 1}I_{k}$ is $V$--dense in $I$ which is $V$--dense. Let $\eps > 0$ and $r \in R$. Because $I$ is $V$--dense, there exists some $a \in I$ such that $V(r - a) \leq \frac{1}{2}\eps$. Moreover, since $\bigcap_{k = 1}^{n + 1}I_{k}$ is $V$--dense in $I$, there exists some $x \in \bigcap_{k = 1}^{n + 1}I_{k}$ such that $V(a - x) \leq \frac{1}{2}\eps$. Thus:
		\[
		V(r - x) = 
		V(r - a + a - x) \leq 
		V(r - a) + V(a - x) \leq 
		\frac{1}{2}\eps + \frac{1}{2}\eps \leq
		\eps.\]
		This proves that $\bigcap_{k = 1}^{n + 1}I_{k}$ is $V$--dense in $R$.
		\finishfullproof
	\end{proof}
	\begin{lemma} \label{lemma:Compact_Product_Converges}
		Let $R$ be a supercomplete, pseudo--valuated ring and let $\mathcal{I}$ be a compact family of closed ideals. Consider the directed family $D$ of all finite subsets of $\mathcal{I}$ and the descending net $\eta\colon D \to \mathcal{L}(R)$ defined by $\eta(\left\{\, I_{1}, \dotsc, I_{n} \,\right\}) := \bigcap\limits_{k = 1}^{n} I_{k}$. If $\mathcal{I}$ is pairwise TCM, then this net converges to $\bigcap\limits_{I \in \mathcal{I}} I$.
	\end{lemma}
	\begin{proof}
		Since $R$ is supercomplete, by Lemma ~\ref{lemma:Hyperspace_Limit_Of_Monotone_Sequence} it is sufficient to show that $\eta$ is a Cauchy net in $\mathcal{L}(R)$.
		Let $V \in \mathcal{V}$ be a pseudo--valuation on $R$ and $\eps > 0$. We need to find $F_{0} \in D$ such that if $F_{0} \subseteq F \in D$ then $(\eta(F_{0}), \eta(F)) \in H(B_{V}(\eps))$. 
		
		By Corollary~\ref{corol:Pairwise_TCM_Families}, there exists a finite $F_{0} \subseteq \mathcal{I}$ such that any $I \in \mathcal{I}$ not belonging to $F_{0}$ is $(V, \frac{1}{2})$--dense.
		We will prove that this is the $F_{0}$ we are looking for. Suppose that $F_{0} \subseteq F \in D$. Since $\eta(F_{0}) \supseteq \eta(F)$, it is enough to prove that
		\[ \eta(F_{0}) \subseteq \eta(F) + B_{V}(\eps),\]
		meaning that $\eta(F)$ is $(V, \eps)$--dense in $\eta(F_{0})$.
		We write $\widehat{F} := F \setminus F_{0}$. 
		Notice that $\widehat{F}$ is pairwise TCM and its elements are $(V, \frac{1}{2})$--dense and therefore $\eta\big(\widehat{F}\big) = \bigcap_{I \in \widehat{F}} I$ is $V$--dense by Corollary~\ref{corol:co_maximal_product}.
		By Lemma~\ref{lemma:TCM_Properties}, $\eta(F_{0})$ and $\eta\big(\widehat{F}\big)$ are TCM so we can apply Lemma~\ref{lemma:co_maximal_product} to conclude that $\eta(F_{0}) \cap \eta\big(\widehat{F}\big)$ is $V$--dense in $\eta(F_{0})$. This completes the proof since $\eta(F) = \eta(F_{0}) \cap \eta\big(\widehat{F}\big)$.
	\end{proof}
	
	\begin{lemma} \label{lemma:Compact_Intersection_Continuity}
		Let $R$ be a supercomplete, pseudo--valuated ring and let $\mathcal{I}$ be a compact family of two--sided ideals of $R$.
		If $\mathcal{I}$ is pairwise co--maximal, then for any $\eps \in \neighbor[R]{0}$ there exists $\delta \in \neighbor[R]{0}$ such that $\bigcap \limits_{I \in \mathcal{I}} (I + \delta) \subseteq \big(\bigcap \limits_{I \in \mathcal{I}} I\big) + \eps$.
	\end{lemma}
	\begin{proof}
		Let $\eps$ be a neighborhood of zero in $R$. By Lemma~\ref{lemma:Compact_Product_Converges}, there exist \linebreak $I_{1}, \dotsc, I_{n} \in \mathcal{I}$ such that $\bigcap\limits_{k = 1}^{n} I_{k} \subseteq \bigcap\limits_{I \in \mathcal{I}} I + \frac{1}{2} \eps$. By Lemma~\ref{lemma:Finite_Intersection_Continuity}, there exists $\delta \in \neighbor[R]{0}$ such that
		$\bigcap \limits_{k = 1}^{n} (I_{k} + \delta) \subseteq \big( \bigcap \limits_{k = 1}^{n} I_{k} \big) + \frac{1}{2}\eps$. 
		Then we obtain:
		\[ \bigcap\limits_{I \in \mathcal{I}} (I + \delta) \subseteq
		   \bigcap\limits_{k = 1}^{n} (I_{k} + \delta) \subseteq
		   \big(\bigcap\limits_{k = 1}^{n} I_{k} \big) + \frac{1}{2}\eps \subseteq
		   \big(\bigcap\limits_{I \in \mathcal{I}} I\big) + \frac{1}{2} \eps + \frac{1}{2}\eps \subseteq
		   \big(\bigcap\limits_{I \in \mathcal{I}} I\big)+ \eps,\]
		which completes the proof.
	\end{proof}
	
	\subsection{The Hyperspace of $A(\Omega)$}
	One important example of a topological ring is the ring $A(\Omega)$ of analytic functions on some domain $\Omega \subseteq \C$ with the compact--open topology. 
	
	\begin{lemma} \label{lemma:Analytic_Basic_Properties}
		Let $\Omega$ be an open subset of $\C$ and let $R$ be the topological ring of all analytic functions over $\Omega$ equipped with the compact--open topology. Then
		\ben 
			\item $R$ is metrizable. \label{lemma:Analytic_Basic_Properties:metric}
			\item $R$ is a complete topological ring. \label{lemma:Analytic_Basic_Properties:complete}
			\item $R$ is pseudo--valuated. \label{lemma:Analytic_Basic_Properties:pseudo_valuated}
		\een
	\end{lemma}
	\begin{proof}~\\
		\ben
			\item We will show that $C(\Omega, \C)$ is metrizable, thus showing that $R$ is also metrizable as a subspace. By~\cite[Theorem~6.7\fullproof on p.~50\finishfullproof]{Warner}, it is enough to show that $C(\Omega, \C)$ has a countable fundamental system of neighborhoods of zero. By Lemma~\ref{lemma:Compact_Coverage_Of_Open_Set}, there exists an ascending sequence $\left\{\, K_{n} \,\right\}_{n \in \N}$ of compact subsets of $\Omega$ such that any compact subset of $\Omega$ is contained in some $K_{n}$. Define:
			\[ U_{n} := \left\{\, f \in C(\Omega, \C) \mid \forall z \in K_{n}: \lvert f(z) \rvert < \frac{1}{n} \,\right\} \]
			for all $n \in \N$. Note that $\left\{\, U_{n} \,\right\}_{n \in \N}$ is a countable fundamental system of neighborhoods of zero.
			\fullproof
			Suppose that $K$ is any compact subset of $\Omega$ and $\eps > 0$. We need to find 
			\[
			U_{n} \subseteq
			\left\{\, f \in C(\Omega, \C) \mid \forall x \in K : |f(x)| \leq \eps  \,\right\}
			\].
			By the choice of $\left\{\, K_{n} \,\right\}_{n \in \N}$, there exists some finite $n_{1} \in \N$ such that $K \subseteq K_{n_{1}}$. Moreover, there exists some $n_{2} \in \N$ such that $\frac{1}{n_{1}} \leq \eps$. Writing $n := \max \left\{\, n_{1}, n_{2} \,\right\}$ we get the desired result.
			\finishfullproof
			\item By~\cite[Theorem~3.5.1\fullproof on p.~88\finishfullproof]{RESG}, if a sequence of $R$ converges to an element $f \in C(\Omega, \C)$, then $f \in R$. Since $C(\Omega, \C)$ is metrizable by (\ref{lemma:Analytic_Basic_Properties:metric}), we conclude that $R$ is a closed subspace of $C(\Omega, \C)$. The space $C(\Omega, \C)$ is complete by~\cite[Theorem~12 on p.~231]{Kelley} so 
			$R$ is complete as a closed subset of a complete set.
			\item By Example~\ref{example:Function_are_pseudo_valuated}, $C(\Omega, \C)$ is pseudo--valuated and therefore $R$ is also pseudo--valuated as one of its subrings.
		\een
	\end{proof}
	
	\begin{example} 
		Let $R = A(\Omega)$ be the topological ring of complex analytic functions over the domain $\Omega \subseteq \C$ with the compact--open topology. Given $z_{0} \in \Omega$, let $I := \left\{\, f \in R|f(z_{0}) = 0  \,\right\}$. The descending sequence of ideals $S_{n} := I^{n}$ does \textbf{not} converge. 
	\end{example}
	\begin{proof}
		If $S_{n}$ converges, then it converges to $\bigcap_{n \in \N} S_{n}$ (by Lemma~\ref{lemma:Hyperspace_Limit_Of_Monotone_Sequence}).
		Choose $z_{0} \neq z_{1} \in \C$ and define $J :=\left\{\, f \in R \mid f(z_{1}) = 0 \,\right\}$. Then $J$ is a maximal ideal and for any $n \in \N$, $S_{n} \nsubseteq J$ , hence $J \bot S_{n}$ for any natural $n$. Note that
		\[\bigcap_{n \in \N} S_{n} = \left\{\, f \in R| \ \forall \ m \in \N: f^{(m)}(z_{0}) = 0  \,\right\} = \left\{\, 0 \,\right\}.\]
		The last equality is due to the fact that a non zero analytic function can only have roots of finite order. 
		If $\left\{\, S_{n} \,\right\}_{n}$ had converged to $\bigcap_{n \in \N} S_{n}$, then $J \bot \bigcap_{n \in \N} S_{n}$ would imply that $J \bot \left\{\, 0 \,\right\}$ by Corollary~\ref{corol:TCM_Is_Closed}. This is clearly not the case since $J$ is a proper closed ideal.
		Thus $\left\{\, S_{n} \,\right\}_{n \in \N}$ does not converge to $\bigcap_{n \in \N} S_{n}$ and therefore it does not converge at all. 
	\end{proof}
	
	\begin{defi}
		Let $\Omega \subseteq \widehat{\C}$ be an open set. A subset $A \subseteq \Omega$ is \emph{faithful to $\Omega$} if any connected component of $\widehat{\C} \setminus A$ intersects $\widehat{\C} \setminus \Omega$.
	\end{defi}
	We will later see that any compact subset of an open $\Omega \subseteq \C$ is contained in another compact subset faithful to $\Omega$.
	\begin{lemma} \label{lemma:Converging_Analytic_Ideals}
		Let $\Omega \subseteq \C$ be a domain and $\left\{\, z_{n} \,\right\}_{n \in \N} \subseteq \Omega$ be a sequence having no accumulation points in $\Omega$. Consider the topological ring $R = A(\Omega)$ with the compact--open topology. Given a corresponding sequence of integers $\left\{\, m_{n} \,\right\}_{n \in \N} \subseteq \N$, we define $I_{n} := \left\langle z - z_{n} \right\rangle ^{m_{n}} \unlhd R$. Then $\lim \limits_{n \to \infty} I_{n} = R$.
	\end{lemma}
	\begin{proof}
		Let $K \subseteq \Omega$ be a compact subset and $\eps > 0$. Let $U \in \neighbor[R]{0}$ be the neighborhood of all analytic functions on $\Omega$ such that their values on $K$ are bounded by $\eps$. We need to find $n_{0} \in \N$ such that for all $n \geq n_{0}$, $(I_{n}, R) \in H(U)$. Since $I_{n} \subseteq R$ for all $n \in \N$, it is sufficient to prove that $R \subseteq I_{n} + U$, meaning that for any $f \in R$ there exists $\tilde{f} \in I_{n}$ such that:
		\[\forall z \in K: |f(z) - \tilde{f}(z)| < \eps. \]
		
		Without loss of generality, we can assume that $K$ is faithful to $\Omega$ (otherwise, we use Lemma~\ref{lemma:Compact_Set_Extension} and replace $K$ with $\hat{K}$). Any infinite subset of a compact set of $\C$ has a limit point
		so any subset of a compact set without a limit point has to be finite. We know that $\left\{\, z_{n} \,\right\}_{n \in \N} \subseteq \Omega$ has no limit points and therefore $\left\{\, z_{n} \,\right\}_{n \in \N} \cap K$ has no limit points in $K$. This implies that $\left\{\, z_{n} \,\right\}_{n \in \N} \cap K$ is finite, or equivalently, that there exists $n_{0} \in \N$ such that for all $n \geq n_{0}$, $z_{n} \notin K$. 
		We claim that this is the $n_{0}$ we were looking for.
		
		To prove this let $f \in R$ be any analytic function on $\Omega$ and $n \geq n_{0}$. Consider the function 
		$g(z) := \dfrac{f(z)}{(z - z_{n})^{m_{n}}}$ which is analytic on $\Omega \setminus \left\{\, z_{n} \,\right\}$. 
		In particular, it is analytic on some neighborhood of $K$.
			
		Since $K$ is compact, we can define $M := \sup \limits_{z \in K} |(z - z_{n})^{m_{n}}|$.  
		We also define 
		\[\conn(\C \setminus K) = \left\{\, C_{\lambda}  \,\right\}_{\lambda \in \Lambda}.\]
		
		Because every connected component of $\hat{\C} \setminus K$ intersects $\hat{\C} \setminus \Omega$, we can pick a number $w_{\lambda} \in C_{\lambda} \cap (\C \setminus \Omega)$ for every $\lambda \in \Lambda$.
		By Runge theorem~\cite[Theorem~12.1.1\fullproof on p.~363 \finishfullproof]{RESG}, there exists a rational function $h$ whose poles are inside $\left\{\, w_{\lambda} \,\right\}_{\lambda \in \Lambda} \subseteq \hat{\C} \setminus \Omega$, so $h$ is analytic on $\Omega$ and
		\[\forall z \in K: |h(z) - g(z)| = \left|h(z) - \dfrac{f(z)}{(z - z_{n})^{m_{n}}}\right| < \frac{\eps}{M}\]
		\[\Downarrow\]
		\[\forall z \in K: |h(z)(z - z_{n})^{m_{n}} - f(z)| < \frac{\eps}{M} \sup \limits_{z \in K} |(z - z_{n})^{m_{n}}| = \eps.\]
		We define $\tilde{f}(z) := h(x)(z - z_{n})^{m_{n}} \in I_{n}$. Hence $\tilde{f}$ approximates $f$ on $K$ proving that $\lim \limits_{n \to \infty} I_{n} = R$.
	\end{proof}

	\begin{corol} \label{corol:Analytic_Zeros_Are_Compact}
		Let $\Omega \subseteq \C$ be a domain and $\left\{\, z_{n} \,\right\}_{n \in \N} \subseteq \Omega$ be a sequence having no accumulation points in $\Omega$. Consider the topological ring $R = A(\Omega)$ with the compact--open topology. For a given sequence of integers $\left\{\, m_{n} \,\right\}_{n \in \N} \subseteq \N$, we define $I_{n} := \left\langle z - z_{n} \right\rangle ^{m_{n}} \unlhd R$. Then $\mathcal{I} := \left\{\, I_{n} \,\right\}_{n = 1}^{\infty} \cup \left\{\, R \,\right\}$ is compact in $\mathcal{L}(R)$.
	\end{corol}
	\begin{proof}
		Recall that by Lemma~\ref{lemma:Analytic_Basic_Properties} $R$ is metrizable. This implies that $\mathcal{I} \subseteq \mathcal{L}(M)$ is also metrizable because of Lemma~\ref{lemma:Hyperspace_Similarities}. In the virtue of~\cite[Theorem~5 on p.~138]{Kelley}, $\mathcal{I}$ is compact if and only if any sequence has a converging subsequence.
		
		Let $\left\{\, J_{n} \,\right\}_{n \in \N}$ be a sequence in $\mathcal{I}$. We will prove that it has a converging subsequence. If $\left\{\, J_{n} \,\right\}_{n \in \N}$ is a finite subset of $\mathcal{I}$, then there is some $J \in \mathcal{I}$ which appears in the sequence infinitely many times. This means that $\left\{\, J_{n} \,\right\}_{n \in \N}$ contains a constant, and therefore converging, subsequence. 
		On the other hand, if $\left\{\, J_{n} \,\right\}_{n \in \N}$ is an infinite subset of $\mathcal{I}$, then it contains a subsequence $\left\{\, J_{n_{k}} \,\right\}_{k \in \N}$ of distinct elements in $\mathcal{I} \setminus \left\{\, R \,\right\}$. Without loss of generality, we can assume that $J_{n_{k}} = I_{n_{k}}$. Note that $\left\{\, z_{n_{k}} \,\right\}_{k \in \N}$ also has no limit points in $\Omega$ so by Lemma~\ref{lemma:Converging_Analytic_Ideals}, $J_{n_{k}}$ converges to $R \in \mathcal{I}$. In either case, $\left\{\, J_{n} \,\right\}_{n \in \N}$ possesses a converging subsequence and therefore $\mathcal{I}$ is compact.
	\end{proof}

	\section{The Chinese Remainder Approximation Theorem} \label{sec:CRAT}
	Throughout this section, let $R$ be a topological ring and let $\mathcal{I}$ be a compact family of pairwise TCM, two--sided ideals. We will denote the ring $\prod\limits_{I \in \mathcal{I}} \quot{R}{I}$ as $\Pi$ and the map sending $r \in R$ to $(r + I)_{I \in \mathcal{I}} \in \Pi$ as $\varphi$. Note that $\varphi$ is clearly a ring homomorphism whose kernel is $\bigcap_{I \in \mathcal{I}} I$. We will also denote by $\widehat{\varphi}$ the algebraic isomorphism from $\quot{R}{\ker \varphi}$ to $\varphi(R)$ obtained from $\varphi$ using the first isomorphism theorem.
	
	The Chinese Remainder Theorem states that when $\mathcal{I}$ is finite and pairwise co--maximal, $\varphi\colon R \to \Pi$ is surjective. Theorem~\ref{thm:Finite_CRAT} could easily be thought of as its immediate topological adaptation. In this section we will continue strengthening this result for the case of infinite, compact families of ideals. To do so, we need to first decide how to topologize the product. The Tychonoff topology might be the first option to come to mind but apparently, better results can be achieved with the following, stronger topology.
	
	\begin{lemma} \label{lemma:Compact_Product_Is_A_Ring}
		For any neighborhood of zero in $R$, say $\eps$, we define 
		\[U(\eps) := \{(a_{I} + I)_{I \in \mathcal{I}} \mid \forall I \in \mathcal{I}: a_{I} \in I + \eps \}.\]
		Then the family $\mathcal{U} := \left\{\, U(\eps) \,\right\}_{\eps \in \neighbor[R]{0}}$ is a fundamental system of neighborhoods of zero for a Hausdorff ring topology on $\Pi$.
	\end{lemma}
	\begin{proof}
		By~\cite[Theorem~3.5\fullproof on pp.~21--22\finishfullproof]{Warner}, we need to prove the following:
		\ben	
			\item For any $U \in \mathcal{U}$ there exists $V \in \mathcal{U}$ such that $2 V \subseteq U$. \label{lemma:Compact_Product_Is_A_Ring:Sum}
			\item For any $U \in \mathcal{U}$ there exists $V \in \mathcal{U}$ such that $-V \subseteq U$.	\label{lemma:Compact_Product_Is_A_Ring:Negative}
			\item For any $U \in \mathcal{U}$ there exists $V \in \mathcal{U}$ such that $V^{2} \subseteq U$. \label{lemma:Compact_Product_Is_A_Ring:Power}
			\item For any $U \in \mathcal{U}$ and any $\overline{r} \in \Pi$, there exists $V \in \mathcal{U}$ such that $\overline{r} V \subseteq U$ and $V \overline{r} \subseteq U$. \label{lemma:Compact_Product_Is_A_Ring:Product}
		\een
		
		Consider $U(\eps) \in \mathcal{U}$.
		Clearly, (\ref{lemma:Compact_Product_Is_A_Ring:Sum}) and (\ref{lemma:Compact_Product_Is_A_Ring:Negative}) are achieved by choosing $V := U(\frac{1}{2} \eps)$ and $V := U(-\eps)$ respectively.
		To prove (\ref{lemma:Compact_Product_Is_A_Ring:Power}), we simply take $V := U(\delta)$ for some neighborhood $\delta$ such that $\delta ^{2} \subseteq \eps$.
		Finally, let $\overline{r} := (r_{I} + I)_{I \in \mathcal{I}}$ be any element of $\Pi$.
		By Corollary~\ref{corol:Pairwise_TCM_Families}, there exists a finite family $\mathcal{F} \subset \mathcal{I}$ such that
		\[\forall I \in \mathcal{I} \setminus \mathcal{F}: R = I + \eps. \]
		Now, for any $I \in \mathcal{F}$, there exists a neighborhood of zero $\delta_{I}$ such that $r_{I} \delta_{I} \subseteq \eps$ and $\delta_{I} r_{I} \subseteq \eps$. We define $\delta := \bigcap\limits_{I \in \mathcal{F}} \delta_{I}$ and $V := U(\delta)$. Note that for any $(s_{I} + I)_{I \in \mathcal{I}} \in V$, $s_{I} \in I + \delta$ and therefore
		\[\forall I \in \mathcal{F}: 
			r_{I} s_{I}  + I \subseteq 
			r_{I} \delta + I \subseteq
			r_{I} \delta_{I} + I \subseteq
			I + \eps \]
		and similarly, $ s_{I} r_{I} + I \subseteq I + \eps$. Moreover, if $I \in \mathcal{I} \setminus \mathcal{F}$, then $R = I + \eps$ implies that $r_{I} \delta + I \subseteq I + \eps$.
		Thus, $V \overline{r} \subseteq U(\eps)$ and $\overline{r} V \subseteq U(\eps)$ which concludes the proof of (\ref{lemma:Compact_Product_Is_A_Ring:Product}).		
	\end{proof}
	From now on, we will always assume that $\Pi$ is furnished with the topology described in Lemma~\ref{lemma:Compact_Product_Is_A_Ring}.\\
	
	We are ready to prove a version of the CRT for infinite amount of ideals in general topological rings.
	\begin{thm}[The Chinese Remainder Approximation Theorem] \label{thm:CRAT}
		If $\mathcal{I}$ is a (possibly infinite) compact family of pairwise TCM two--sided ideals, then $\varphi$ is continuous and its image is dense in $\Pi$.
	\end{thm}
	\begin{proof}
		To prove the continuity of $\varphi$, let $\eps$ be a neighborhood of zero in $R$. We need to find a neighborhood $\delta$ such that $\varphi(\delta) \subseteq U(\eps)$.
		Note that 
		\[\varphi(\eps) = \{(r + I)_{I \in \mathcal{I}} \mid r \in \eps \} 
		  \subseteq \{(r_{I} + I)_{I \in \mathcal{I}} \mid \forall I \in \mathcal{I}: r_{I} \in \eps\} = U(\eps),\]
		and therefore we can simply choose $\delta := \eps$.
		
		To prove that the image of $\varphi$ is dense in $\Pi$, let $\eps$ be a symmetric neighborhood of zero in $R$ and let $\overline{r} = (r_{I} + I)_{I \in \mathcal{I}}$ be an element of $\Pi$. We need to find $r \in R$ such that $\varphi(r) \in \overline{r} + U(\eps)$.
		By Corollary~\ref{corol:Pairwise_TCM_Families}, there exists a finite $\mathcal{F} \subseteq \mathcal{I}$ such that
		\[\forall I \in \mathcal{I} \setminus \mathcal{F}: R = I + \eps.\]
		By Theorem~\ref{thm:Finite_CRAT}, there exists $r \in R$ such that
		\[\forall I \in \mathcal{F}: r + I \subseteq r_{I} + I + \eps. \]
		Moreover, for any $I \in \mathcal{I} \setminus \mathcal{F}$, $R = I + \eps$ implies $r + I \subseteq r_{I} + I + \eps$. Thus, for any $I \in \mathcal{I}$, $r + I \subseteq r_{I} + I + \eps$, or in other words $\varphi(r) \in \overline{r} + U(\eps)$, as required.
	\end{proof}
	
	\begin{corol} \label{corol:Compact_CRAT}
		If $R$ is compact, then for any family $\mathcal{I}$ of pairwise TCM two--sided ideals, $\widehat{\varphi}\colon \quot{R}{ker \varphi} \to \Pi$ 
		is a topological isomorphism.
	\end{corol}
	\begin{proof}
		Without loss of generality, we can assume that $R \in \mathcal{I}$ (since $\quot{R}{R}$ is trivial). We claim that $\mathcal{I}$ is a closed subspace of $\mathcal{L}(R)$. By Lemma~\ref{lemma:pairwise_TCM_closure}, $\overline{\mathcal{I}}$ is pairwise TCM. Suppose that $I \in \overline{\mathcal{I}}$. If $I = R$ then $I \in \mathcal{I}$. Otherwise, by Corollary~\ref{corol:Pairwise_TCM_Families}, $I$ is an isolated point in $\overline{\mathcal{I}}$ and therefore an element of $\mathcal{I}$. This proves our claim.
		
		Since $R$ is compact, so is $\mathcal{L}(R)$ (by Lemma~\ref{lemma:Hyperspace_Similarities}). Then $\mathcal{I}$, as a closed subspace of $\mathcal{L}(R)$, is also compact. Thus, we can apply Theorem~\ref{thm:CRAT} to conclude that the image of $\varphi$ is dense in $\prod$. Moreover, $\varphi$ is continuous and therefore its image is compact too. In virtue of $\varphi(R)$ being both dense and compact, $\varphi$ is surjective. Finally, since $R$ is compact, $\varphi$ is both closed and open. By~\cite[Theorem~5.11\fullproof on p.~35\finishfullproof]{Warner}, $\varphi$ being open implies that $\widehat{\varphi}$ is a topological isomorphism.
	\end{proof}
	\begin{lemma} \label{lemma:cannonical_homo_is_open}
		If $R$ is supercomplete and pseudo--valuated then $\varphi\colon R \to \Pi$ is open onto its image.
	\end{lemma}
	\begin{proof}
		Let $\eps$ be a symmetric neighborhood of zero in $R$. We will prove that $\varphi(\eps)$ is a neighborhood of zero in $\varphi(R)$. By Lemma~\ref{lemma:Compact_Intersection_Continuity}, there exists a neighborhood $\delta \in \neighbor[R]{0}$ such that 
		$\bigcap\limits_{I \in \mathcal{I}} (I + \delta) \subseteq \bigcap\limits_{I \in \mathcal{I}} I + \eps$.
		We define $U = U(\delta) \cap \varphi(R)$. We will show that $U \subseteq \varphi(\eps)$ thus proving our claim. Suppose that $(r + I)_{I \in \mathcal{I}}$ is an element of $U$. Note that $r \in \bigcap_{I \in \mathcal{I}} (I + \delta)$ and therefore $r \in \bigcap_{I \in \mathcal{I}} I + \eps$ meaning that $\varphi(r) \in \varphi(\bigcap\limits_{I \in \mathcal{I}} I) + \varphi(\eps) = \varphi(\eps)$, as required.
	\end{proof}
	
	Our last theorem in this section will be a stronger version of the previous theorem for supercomplete pseudo--valuated rings. 
	\begin{thm} \label{thm:CRAT_For_Supercomplete_Pseudovaluated_Rings}
		If $R$ is supercomplete and pseudo--valuated, then $\widehat{\varphi}\colon \quot{R}{\ker \varphi} \to \Pi$ is a topological isomorphism onto $\Pi$. 
	\end{thm}
	\begin{proof}
		By Lemma~\ref{lemma:cannonical_homo_is_open}, $\varphi$ is open and by Theorem~\ref{thm:CRAT} it is continuous. By Lemma~\ref{lemma:Image_Of_Supercomplete_Ring}, the image of $\varphi$ is supercomplete. In particular, $\varphi(R)$ is a complete subset of $\Pi$ so it is closed in $\Pi$ by~\cite[Theorem~22 on p.~192]{Kelley}. However, $\varphi(R)$ is dense in $\Pi$ by Theorem $\ref{thm:CRAT}$ so $\varphi(R) = \Pi$. Applying~\cite[Theorem~5.11\fullproof on p.~35\finishfullproof]{Warner} like we did in Corollary~\ref{corol:Compact_CRAT} we get the desired result. 
	\end{proof}
	\section{Applications}
	One well--known result proven using CRT is the Lagrange interpolation (\cite[Corollary~9 on p.~366]{Childs}): Let $\F$ be a field and let $x_{0}, \dotsc, x_{n}$ be distinct elements of $\F$. For any $y_{0}, \dotsc, y_{n} \in \F$ there exists a unique polynomial $p(x) \in \F[x]$ of degree less or equal to $n$ such that $p(x_{i}) = y_{i}$ for each $i = 0, \dotsc , n$.
	
	We will now prove a very similar result, for a larger class of topological rings. Namely, topologically simple rings.
	\begin{defi}
		A topological ring $S$ is \emph{topologically simple} if it has no closed ideals other than $S$ itself and $\left\{\, 0 \,\right\}$.
	\end{defi}
	\begin{example}
		Any dense subring of a topologically simple ring is topologically simple~\cite[Proposition~1.4.16\fullproof on p.~58\finishfullproof]{VISTAV}. This means that every dense subring of a field is topologically simple. For example, $\Z[\sqrt{2}]$ is dense in $\R$ and therefore topologically simple.
	\end{example}
	\begin{thm} \label{thm:Lagrange_Interpolation}
		Let $S$ be a topologically simple ring. Let $\left\{\, x_{i} \,\right\}_{i = 1}^{n} \subseteq S$ be distinct points and $\left\{\, y_{i} \,\right\}_{i=1}^{n} \subseteq S$ be some set of values. For every $\eps \in \neighbor{0}$, there exists a polynomial $p \in S[x]$ such that:
		\[p(x_{i}) - y_{i} \in \eps\]
		for any $1 \leq i \leq n$.
	\end{thm}
	\begin{proof}
		Consider $S[x]$ with the pointwise topology. For any $1 \leq i \leq n$, we define:
		\[J_{i} := \left\{\, p \in S[x] \mid p(x_{i}) = 0 \,\right\} = \langle x - x_{i}\rangle.\]
		
		For any $1 \leq i \neq j \leq n$:
		\[
		\overline{J_{i} + J_{j}} =
		\overline{\langle x - x_{i} \rangle + \langle x - x_{j}\rangle} \supseteq
		\overline{\langle (x - x_{i}) - (x - x_{j})\rangle} = 
		\overline{\langle x_{j} - x_{i} \rangle}.
		\]
		Since $S$ it topologically simple and $x_{j} - x_{i} \neq 0$, the ideal $\overline{\langle x_{j} - x_{i} \rangle}$ contains $S$ and in particular, $1 \in \overline{\langle x_{j} - x_{i} \rangle}$. This implies that $\overline{\langle x_{j} - x_{i} \rangle} = S[x]$.
		\fullproof
		Note that the topology of $S$ as a subset of $S[x]$ coincides with that of $S$ making our last argument valid.
		\finishfullproof
		By the CRAT (Theorem~\ref{thm:CRAT}), the map  $\varphi\colon S[x] \to \prod_{i=1}^{n}\quot{S[x]}{J_{i}}$ defined by:
		\[\varphi(p(x)) := (p(x) +  J_{i})_{i = 1}^{n}\]
		is dense onto its image.
		
		Consider the neighborhood
		$U := 
		\prod\limits_{i = 1}^{n} (y_{i} + J_{i} + \eps)$ of the element $(y_{i} + J_{i})_{i = 1}^{n} \in \prod_{i=1}^{n}\quot{S[x]}{J_{i}}$.
		Because $\varphi(S[x])$ is dense in $\prod_{i=1}^{n}\quot{S[x]}{J_{i}}$, there exists a polynomial $p \in S[x]$ such that $\varphi(p) \in U$. Note that $\varphi(p) = (p(x_{i}) + J_{i})_{i = 1}^{n}$ and therefore 
		\[
		p(x_{i}) + J_{i} \subseteq y_{i} + J_{i} + \eps
		\]
		for any $1 \leq i \leq n$.
		It immediately follows that: 
		\[ 
		\forall 1 \leq i \leq n: p(x_{i}) \in y_{i} + \eps.
		\]
		\fullproof
		This is because there exists $q \in J_{i}$ such that 
		\[
		p(x_{i}) + q \in y_{i} + \eps.
		\]
		By definition, $q(x_{i}) = 0$ so $p(x_{i}) \in y_{i} + \eps$.
		\finishfullproof
		This finishes the proof.
	\end{proof}
	We show a new proof for another known interpolation result.
	\begin{thm}~\cite[Theorem~15.13\fullproof on pp.~304--305\finishfullproof]{Rudin}
		Suppose that $\Omega$ is an open region in $\C$ and $\left\{\, z_{n} \,\right\} \subseteq \Omega$ is a sequence of distinct points having no limit point in $\Omega$.
		For any sequence of integers $\left\{\, m_{n} \,\right\}_{n \in \N}$ and numbers $w_{1, n}, \dotsc , w_{m_{n}, n} \in \C$ there exists an analytic function $f$ on $\Omega$ such that:
		\[\forall n \in \N, 1 \leq k \leq m_{n}: f^{(k)}(z_{n}) = k! \ w_{k, n}. \]
	\end{thm}
	\begin{proof}
		Consider the complete, metric, pseudo--valuated ring $R = A(\Omega)$ (Lemma~\ref{lemma:Analytic_Basic_Properties}) and the closed ideals $I_{n} = \langle z - z_{n} \rangle^{m_{n} + 1}$. By Corollary~\ref{corol:Analytic_Zeros_Are_Compact}, the family \linebreak $\mathcal{I} := \left\{\, I_{n} \,\right\}_{n \in \N} \cup  \left\{\, R \,\right\}$ is compact. Note that $\quot{R}{R}$ is trivial so we will ignore it in the product. We define the following element of $\Pi$:
		\[\overline{f} = \big(\sum\limits_{k = 1}^{m_{n}} w_{k, n} \ (z - z_{n})^{k} + I_{n}\big)_{n \in \N} \]
		Note that $R$ is metrizable and therefore supercomplete by~\cite[Theorem 48 on p.~30]{Isbell}. Theorem~\ref{thm:CRAT_For_Supercomplete_Pseudovaluated_Rings} implies that there exists an analytic function $f \in R$ such that $\varphi(f) = \overline{f}$. Thus, for all $n \in \N$:
		\[f(z) - \sum\limits_{k = 1}^{m_{n}} w_{k, n} \ (z - z_{n})^{k} \in I_{n} =  \langle z - z_{n}\rangle^{m_{n} + 1}.\]
		In other words, there exists an analytic function $g \in R$ such that:
		\[f(z) = \sum\limits_{k = 1}^{m_{n}} w_{k, n} \ (z - z_{n})^{k} + g(z)(z - z_{n})^{m_{n} + 1}.\]
		Calculating the derivatives of both sides we get:
		\[f^{(k)}(z_{n}) = k! \ w_{k, n}, \]
		for all $n \in \N$ and $1 \leq k \leq m_{n}$.
	\end{proof}
	\appendix
	
	\section{Hyperspace}
	
	\begin{lemma} \label{lemma:Hyperspace_Multipication_Is_Continuous}
		Let $R$ be a topological ring. For every $r \in R$, the function \linebreak $m_{r}: H(R) \rightarrow H(R)$ defined by $m_{r}(A) := \overline{r A}$ is uniformly continuous.
	\end{lemma}
	\begin{proof}
		Let $\eps$ be a neighborhood of zero in $R$. We need to find $\delta$ such that if $(A,B)  \in H(\delta)$ then $(m_{r} (A), m_{r} (B)) \in H(\eps)$. We choose $\delta$ such that $r \delta \subseteq \frac{1}{2}\eps$.
		\fullproof
		(it is possible since multiplication is continuous).
		\finishfullproof
		Now, if $(A,B)  \in H(\delta)$, then by definition $A \subseteq B + \delta$ and $B \subseteq A + \delta$. Multiplying both sides by $r$ and taking the closure yields:
		\ben
			\item $m_{r} (A) = \overline{rA} \subseteq \overline{rB + r\delta} \subseteq \overline{rB + \frac{1}{2}\eps} \subseteq \overline{rB} + \eps = m_{r} (B) + \eps$.
			\item $m_{r} (B) = \overline{rB} \subseteq \overline{rA + r\delta} \subseteq \overline{rA + \frac{1}{2}\eps} \subseteq \overline{rA} + \eps = m_{r} (A) + \eps$.
		\een
		Note that we could add $\frac{1}{2}\eps$ and ignore the closure operation (like we did in Lemma~\ref{lemma:Join_Is_Continuous}) because $\overline{A} = \bigcap \left\{\, A + U \mid U \in N(0) \,\right\}$ (\cite[Theorem~3.3\fullproof on p.~20\finishfullproof]{Warner}).
		Therefore, $(m_{r} (A), m_{r} (B)) \in H(\eps)$, as required.
	\end{proof}
	\begin{lemma} \label{lemma:Modules_Are_Closed}
		If $R$ is a topological ring, then
		$\mathcal{L}(R)$ is a closed subset of the hyperspace of $R$.
	\end{lemma}
	
	\begin{proof}
		Let $\left\{\, I_{\alpha} \,\right\}_{\alpha \in A}$ be a converging net of closed two--sided ideals in the hyperspace. We need to prove that its limit is also a two--sided ideal. Write $I = \lim I_{\alpha}$. By Lemma~\ref{lemma:Join_Is_Continuous}:
		\[I + I \subseteq I \vee I = \lim (I_{\alpha} \vee I_{\alpha}) = \lim \overline{I_{\alpha} + I_{\alpha}}
		= \lim I_{\alpha} = I,\]
		meaning that $I$ is a subgroup. By Lemma~\ref{lemma:Hyperspace_Multipication_Is_Continuous}, for any $r \in R$:
		\[rI \subseteq \overline{rI} = \lim \overline{ r I_{\alpha}} \subseteq \lim \overline{I_{\alpha}} = \lim I_{\alpha} = I.\]
		Thus, $I$ is a left ideal. A similar argument shows that $I$ is also a right ideal.
	\end{proof}
	\begin{defi}\cite[p.~29]{Isbell}
		Let $X$ be a topological space and $\left\{\, Y_{\alpha} \,\right\}_{\alpha \in A}$ be a net of subsets of $X$. A point $x \in X$ is said to be a \emph{cluster point} of $\left\{\, Y_{\alpha} \,\right\}_{\alpha \in A}$ if any neighborhood $U \in \neighbor{x}$ meets $\left\{\, Y_{\alpha} \,\right\}_{\alpha \in A}$ in a cofinal set of indexes.
	\end{defi}
	\begin{lemma} \label{lemma:Hyperspace_Limit_Charecerization}
		~\cite[Proposition~45 on p.~29]{Isbell}
		Let $\mu X$ be a uniform space, we consider its hyperspace $H(X)$ with the hyperspace uniformity. 
		Let $\left\{\, A_{\alpha} \,\right\}_{\alpha \in A}$ be a net in $H(X)$. If $\left\{\, A_{\alpha} \,\right\}_{\alpha \in A}$ converges, it converges to its set of cluster points.
	\end{lemma}
	
	\begin{lemma} \label{lemma:Hyperspace_Limit_Of_Monotone_Sequence}
		Let $R$ be a topological ring and consider its hyperspace $H(R)$ with the induced uniformity. 
		Suppose that $\left\{\, S_{\alpha} \,\right\}_{\alpha \in A}$ is a descending net in $H(R)$. If $\left\{\, S_{\alpha} \,\right\}_{\alpha \in A}$ converges, it converges to $\bigcap \limits_{\alpha \in A} S_{\alpha}$.
	\end{lemma}
	\begin{proof} 
		By Lemma~\ref{lemma:Hyperspace_Limit_Charecerization}, we need to show that the set of cluster points of $\left\{\, S_{\alpha} \,\right\}_{\alpha \in A}$ coincides with the intersection. First we show that the intersection is contained in the set of cluster points $C$. 
		Suppose $x \in \bigcap \limits_{\alpha \in A} \left\{\, S_{\alpha} \,\right\}$. 
		For every $U \in \neighbor{x}$ and $\alpha \in A$,  $x \in U \cap S_{\alpha}$, so the set $\{\alpha \in A \mid U \cap S_{\alpha} \neq \phi\} = A$ is cofinal in $A$. By definition, $x \in C$.
		
		Now we will prove that $C$ is contained in each $S_{\alpha}$. Let $\alpha \in A$. Since $\left\{\, S_{\alpha} \,\right\}_{\alpha \in A}$ converges to $C$, for every neighborhood of zero $\eps$ there exists $\alpha_{\eps}$ such that for any 
		$\beta \geq \alpha_{\eps}: C \subseteq S_{\beta} + \eps$. The set $A$ is directed, hence there exists an upper bound for $\alpha$ and $\alpha_{\eps}$ (call it $\tilde{\alpha}$). The net is descending, so
		\[C \subseteq S_{\tilde{\alpha}} + \eps \subseteq S_{\alpha} + \eps. \]
		This is true for any $\eps$. As in Lemma~\ref{lemma:Join_Is_Continuous}, we apply~\cite[Theorem~3.3\fullproof on p.~20\finishfullproof]{Warner}, namely that $\overline{S_{\alpha}} = \bigcap \left\{\, S_{\alpha} + U \mid U \in \neighbor{0} \,\right\}$, to conclude that: $C \subseteq \overline{S_{\alpha}} = S_{\alpha}$.
		This is true for all $S_{\alpha}$ and therefore $C \subseteq \bigcap \limits_{\alpha \in A} S_{\alpha}$ which completes the proof.
	\end{proof}
	\section{Two properties of $\C$}
	\begin{lemma} \label{lemma:Compact_Set_Extension}
		Let $\Omega \subseteq \C$ be an open set and let $K \subseteq \Omega$ be a compact subset. There exists a compact subset $\hat{K}$ containing $K$ and faithful to $\Omega$.
	\end{lemma}
	\begin{proof}		
		We define
		\[U := \bigcup \left\{\, C \in \conn(\hat{\C} \setminus K)\mid C \cap (\hat{\C} \setminus \Omega) = \phi \,\right\}
		= \bigcup \left\{\, C \in \conn(\hat{\C} \setminus K)\mid C \subseteq \Omega  \,\right\}.\]
		Then $U \subseteq \Omega$. We define $\hat{K} := K \cup U \subseteq \Omega$. 
		
		The Riemann sphere is locally connected, so $\hat{\C} \setminus K$ is also locally connected as an open subset. By~\cite[Corollary~27.10\fullproof on p.~200\finishfullproof]{Willard}, the connected components of a locally connected space are open.
		Note that:
		\[\hat{\C} \setminus \hat{K} =\bigcup \left\{\, C \in \conn(\hat{\C} \setminus K)\mid C \cap (\hat{\C} \setminus \Omega) \neq \phi \,\right\}.\]
		This is a disjoint union of open connected components and therefore:
		\[\conn(\hat{\C} \setminus \hat{K})= \left\{\, C \in \conn(\hat{\C} \setminus K)\mid C \cap (\hat{\C} \setminus \Omega) \neq \phi  \,\right\}. \]
		This means that
		\[\forall C \in \conn(\hat{\C} \setminus \hat{K}): C \cap (\hat{\C} \setminus \Omega) \neq \phi. \]
		By definition, $\hat{K}$ is faithful to $\Omega$.
		All that is left is to prove that $\hat{K}$ is compact.
		
		Notice that
		\[\hat{\C} \setminus \hat{K}=\bigcup \left\{\, C \in \conn(\hat{\C} \setminus K)\mid C \nsubseteq \Omega \,\right\}\]
		is a union of open sets and is therefore open. Thus, $\hat{K}$ is compact as a closed subspace of the Riemann sphere.
	\end{proof}
	
	\begin{lemma} \label{lemma:Compact_Coverage_Of_Open_Set}
		Let $\Omega$ be an open subset of $\C$. There exists an ascending sequence $\left\{\, K_{n} \,\right\}_{n = 1}^{\infty}$ of compact subsets of $\Omega$ such that any compact $K \subseteq \Omega$ is contained in some $K_{n}$.
	\end{lemma} 
	\begin{proof}
		For any $n \in \N$ we define:
		\[ K_{n} := \overline{B(n)} \cap \left\{\, z \in \C \mid d(z, \C \setminus \Omega) \geq \frac{1}{n} \,\right\}, \]
		where $B(n)$ is the open ball with radius $n$ around $0$ in $\C$.
		Note that $K_{n}$ is clearly closed in $\C$ for all $n \in \N$.
		Because $K_{n} \subseteq \overline{B(n)}$, it is also bounded and therefore compact.
		Suppose that $K \subseteq \Omega$ is any compact set. There exists some natural numbers $n_{1}, n_{2} \in \N$ such that:
		\ben
			\item $K \subseteq \overline{B(n_{1})}$.
			\item $\frac{1}{n_{2}} \leq d(K, \C \setminus \Omega)$.
		\een
		Define $n := \max\left\{\, n_{1}, n_{2} \,\right\}$. Clearly, $K \subseteq K_{n}$ as required.
		\fullproof
		The key here is to notice that $f(z) := d(z, \C \setminus \Omega)$ is continuous so it has a minimum value on $K$. By definition, this value is also equals to $d(K, \C \setminus \Omega)$. If it were zero, there would have been a point $z \in K \subseteq \Omega$ such that $d(z, \C \setminus \Omega) = 0$. Because $\Omega$ is open, this is not possible. Thus, $d(K, \C \setminus \Omega) > 0$.
		\finishfullproof
	\end{proof}

\end{document}